\documentclass[12pt,a4paper]{amsart}

\usepackage{color}
\usepackage[pdfstartview=FitH]{hyperref}

\usepackage[T1]{fontenc}
\usepackage{amsmath, amssymb, amsthm}
\usepackage{enumitem}

\newtheorem{thm}{Theorem}[section]

\newtheorem*{yano}{Yano's Conjecture}
\newtheorem{prop}[thm]{Proposition}
\newtheorem{prop0}{Proposition}
\newtheorem{cor}[thm]{Corollary}
\newtheorem{lema}[thm]{Lemma}

\theoremstyle{remark}
\newtheorem{remark}[thm]{Remark}
\theoremstyle{definition}
\newtheorem{dfn}[thm]{Definition}
\newtheorem{ntc}[thm]{Notation}
\newtheorem{ejm}[thm]{Example}

\DeclareMathOperator{\ord}{ord}

\DeclareMathOperator{\spec}{Spec}

\newcommand\bz{{\mathbb Z}}

\newcommand\bc{{\mathbb C}}
\newcommand\br{{\mathbb R}}

\newcommand{\cO}{{\mathcal O}}

\DeclareMathOperator*{\res}{Res}

\newcommand\enet[1]{\renewcommand\theenumi{#1}
\renewcommand\labelenumi{\theenumi}}

\title[Yano's conjecture for 2-Puiseux pairs]{Yano's conjecture for 2-Puiseux pairs irreducible plane curve singularities}

\author[E. Artal]{E.~Artal Bartolo${^1}$}

\address{Departamento de Matem\'aticas-IUMA, Universidad de Zaragoza,
c/ Pedro Cerbuna 12, 50009 Zaragoza, SPAIN}

\email{artal@unizar.es}

\thanks{${^1}$Partially supported by
MTM2013-45710-C02-01-P and Grupo Geometr{\'i}a of Gobierno de Arag{\'o}n/Fondo Social Europeo.}
\author[Pi.~Cassou-Nogu\`es]{Pi.~Cassou-Nogu\`es${^2}$}
\address{Institut de Math\'ematiques de Bordeaux, Universit\'e de Bordeaux , 
350, Cours de la Lib\'eration, 33405, Talence Cedex 05, FRANCE}
\email{Pierrette.Cassou-nogues@math.u-bordeaux.fr}
\thanks{${^2}$Partially supported by 
MTM2013-45710-C02-01-P and MTM2013-45710-C02-02-P}

\author[I. Luengo]{I. Luengo${^3}$}

\author[A. Melle]{A.~Melle-Hern\'andez${^3}$}

\address{ICMAT (CSIC-UAM-UC3M-UCM) \\
Departamento de {\'A}lgebra, 
Facultad de Ciencias Matem{\'a}ticas, Universidad Complutense, 28040 Madrid,  SPAIN}

\curraddr{}
\email{iluengo@ucm.es,amelle@ucm.es}

\thanks{${^3}$Partially
supported by the grant 
MTM2013-45710-C02-02-P}


\date{}

\dedicatory{}

\begin{document}

\begin{abstract}
In 1982, Tamaki Yano proposed a conjecture predicting  
the $b$-exponents of an  irreducible plane curve singularity germ which is generic in its equisingularity class. 
In this article we prove the conjecture for the case in which the irreducible germ has two Puiseux pairs 
and its algebraic monodromy has distinct eigenvalues. 
This hypothesis on the monodromy implies that the $b$-exponents coincide with  the opposite of the roots of the Bernstein polynomial, 
and we compute the roots of the Bernstein polynomial.
\end{abstract}

\maketitle

\section*{Introduction}

The Bernstein polynomial of a singularity germ is a powerful analytic invariant, but it is, in general, extremely hard
to compute, even in the case of irreducible  plane curve singularities. It is well-known
that the Bernstein polynomial vary in the $\mu$-constant stratum of such germs. Since
this stratum is irreducible, it is conceivable that a \emph{generic} Bernstein polynomial
exists, i.e., there exists a dense Zariski-open set in the stratum where the Bernstein polynomial 
remains constant. From the computational point of view it is even harder to effectively compute 
this generic polynomial. In~1982, Tamaki Yano conjectured a closed formula for the Bernstein polynomial of 
an irreducible plane curve which is generic in its equisingularity class, \cite[Conjecture~2.6]{Y82}.
This conjecture is still open. The aim of this paper is to provide a significant progress by proving
it for a \emph{big} family of $2$-Puiseux-pairs singularities.

Let $\cO$ be the ring  of germs of holomorphic functions on $(\bc^n,0)$, 
$\mathcal{D}$ the  ring of germs of holomorphic differential operators of finite order with coefficients in $\cO$. 
Let $s$ be an indeterminate commuting with the elements of $\mathcal{D}$ and set 
$\mathcal{D}[s]=\mathcal{D}\otimes_{\mathbb{C}} \mathbb{C}[s].$
 
Given a holomorphic germ  $f\in \cO$, one considers 
$\cO\left[\frac{1}{f}, s\right] f^s$ as a free $\cO \left[\frac{1}{f}, s\right]$-module of rank $1$  
with the  natural  $\mathcal{D}[s]$-module structure. Then,   
there exits a non-zero polynomial $B(s)\in \bc[s]$ and some 
differential operator $P=P(x,\frac{\partial}{\partial x},s)\in\mathcal{D}[s]$, 
holomorphic in $x_1,\dots,x_n$ and polynomial in 
$\frac{\partial }{\partial x_1},\dots,\frac{\partial }{\partial x_n}$, which  satisfy  in $\cO \left[\frac{1}{f}, s\right] f^s$ the following functional equation
\begin{equation}\label{berstein-rel}
P(s,x,D)\cdot f(x)^{s+1}=B(s)\cdot f(x)^s.
\end{equation}
The monic generator $b_{f,0}(s)$ of the ideal  of such polynomials $B(s)$  is called
the Bernstein polynomial (or $b$-function or  Bernstein-Sato polynomial) of $f$ at $0$.  
The same result holds if we replace $\cO$
by the ring of polynomials in a field ${\mathbb K}$ of  zero characteristic with the obvious corrections,  see e.g. 
\cite[Section~10, Theorem~3.3]{Co95}.

This result was first obtained for $f$ polynomial by  Bernstein in~\cite{B72} and in general 
by Bj\"{o}rk \cite{B:81}.
One can  prove  that $b_{f,0}(s)$ is divisible by $s+1$, and we also consider the reduced Bernstein polynomial  
$\tilde{b}_{f,0}(s):=\frac{b_{f,0}(s)}{s+1}$.

In the case where $f$  defines an isolated singularity, one can consider 
the Brieskorn lattice  
$H_0^{''}:= \Omega^n /df \wedge d \Omega^{n-2}$ and its saturated  
$ \tilde{H}_0^{''}=\sum_ {k\geq 0}  (\partial _t t)^k      H_0^{''} $.  
Malgrange \cite{M:75} showed that  
the reduced Bernstein polynomial  
$\tilde{b}_{f,0}(s)$ is the minimal polynomial of the endomorphism $-\partial_t t$ on the vector space 
$F:=\tilde{H}_0^{''}/ \partial_t^{-1}  \tilde{H}_0^{''}$, whose dimension equals  the
 Minor number  $\mu(f,0)$ of $f$ at $0$. Following Malgrange \cite{M:75},
the set of $b$-\emph{exponents} are  
the $\mu$ roots  $\{\alpha_1, \ldots,\alpha_\mu \}$  of the characteristic polynomial of the   endomorphism $-\partial_t t$.
Recall also that $\exp(-2i\pi\partial_t t)$ can be identified with the (complex) \emph{algebraic monodromy}
of the corresponding Milnor fibre $F_f$ of the singularity at the origin.

Kashiwara~\cite{K:76} expressed these ideas using  differential operators and considered
 $\mathcal{M}:=\mathcal{D}[s]f^s/\mathcal{D}[s]f^{s+1}$,
where $s$ defines an endomorphism of $P(s)f^s$ by multiplication.
This morphism keeps invariant $\tilde{\mathcal{M}}:=(s+1)\mathcal{M}$
and defines a linear endomorphism of $(\Omega^n\otimes_{\mathcal{D}}\tilde{\mathcal{M}})_0$
which is naturally identified with $F$ and under this identification $-\partial_t t$ 
becomes the endomorphism defined by the multiplication by~$s$.

In~\cite{M:75}, Malgrange proved that the set $R_{f,0}$ of roots of the Bernstein polynomial 
is contained in $\mathbb{Q}_{<0}$, see also Kashiwara~\cite{K:76}, who also restricts the set of candidate roots.
The number $-\alpha_{f,0}:=\max R_{f,0}$ 
is the opposite of the log canonical threshold of the singularity and Saito~\cite[Theorem~0.4]{MS94}
proved that 
\begin{equation}\label{eq:saito}
R_{f,0}\subset[\alpha_{f,0}-n,-\alpha_{f,0}].
\end{equation}

Now let  $f$ be an irreducible germ of  plane curve. In 1982, Tamaki Yano \cite{Y82} made a conjecture concerning 
the $b$-exponents of such germs.
Let $(n, \beta_1,\beta_2,\ldots, \beta_g)$  be the characteristic sequence of $f$,  see e.g. \cite[Section 3.1]{W04}. 
Recall that this means that $f(x,y)=0$ has as root (say over~$x$) a Puiseux expansion
$$
x=\dots+a_1 y^{\frac{\beta_1}{n}}+\dots+a_g y^{\frac{\beta_g}{n}}+\dots
$$
with exactly~$g$ characteristic monomials.
Denote $\beta_0:=n$ and
define recursively 
$$
e^{(k)}:=
\begin{cases}
n&\text{ if }k=0,\\
\gcd(e^{(k-1)},\beta_k)&\text{ if }1\leq k\leq g.
\end{cases}
$$
We define the following numbers for $1\leq k\leq g$:
$$
R_k:=\frac{1}{e^{(k)}}\left(\beta_k e^{(k-1)} +
\sum_{j=0}^{k-2}\beta_{j+1}\left(e^{(j)}-e^{(j+1)}\right)\right),\qquad
r_k:= \frac{\beta_k+n}{e^{(k)}}.
$$ 
Note that $R_k$ admits the following recursive formula:
$$
R_k:=
\begin{cases}
n&\text{ if }k=0,\\
\frac{e^{(k-1)}}{e^{(k)}}\left(R_{k-1}+\beta_k-\beta_{k-1}\right)&\text{ if }1\leq k\leq g.
\end{cases}
$$
We end with the following definitions ${R'_0}:=n$,  $r'_0 :=2$ and
for $1\leq k\leq g$:
$$
{R'_k}:=\frac{R_k e^{(k)}}{e^{(k-1)}},\quad
r'_k := \left\lfloor  r_k e^{(k)}/e^{(k-1)} \right\rfloor +1. 
$$   
Yano defined the following polynomial with fractional powers in $t$ 
\begin{equation}\label{eq:generating}
R(n,\beta_1,\ldots,\beta_g;t):=
t+\sum_{k=1}^g t^{\frac{r_k}{R_k}}\frac{1-t}{1-t^{\frac{1}{R_k}}}
-\sum_{k=0}^g t^{\frac{r'_k}{R'_k}}\frac{1-t}{1-t^{\frac{1}{{R'_k}}}},
\end{equation}
and he  proved that  $R(n,\beta_1,\ldots,\beta_g;t)$ has non-negative coefficients.

The number of monomials in 
$R(n,\beta_1,\ldots,\beta_g;t)$ is equal to
$1+\sum_{k=1}^{g} R_k- \sum_{k=0}^g R'_k$
and one can prove that this number is the Milnor number $\mu.$
The numbers $R_k$ (resp. $R'_k$) are the multiplicities of the irreducible exceptional  divisors of the
minimal embedded resolution of the singularity whose smooth part has Euler characteristic~$-1$
(resp.~$1$), see e.g. Lemma~3.6.1, Fig 3.5 and Theorem~8.5.2 in \cite{W04}. 
Using A'Campo formula~\cite{AC75} for the Euler characteristic of the Milnor fibre $F_f$ of $f$ at $0$, that is $1-\mu=\chi(F_f)$, 
one gets $\chi(F_f)=-\sum_{k=1}^g R_k+\sum_{k=0}^g R'_k$, that is that number equals to    
$\mu$. 

\begin{yano}[\cite{Y82}] \hspace*{-2mm}  For almost all irreducible plane curve singularity germ $f:(\bc^2,0)\to (\bc,0)$  with characteristic sequence
$(n, \beta_1,\beta_2,\ldots, \beta_g)$, the $b$-exponents  $\{\alpha_1\!,\ldots,\!\alpha_\mu\}$ 
 are given by  the generating series
\begin{equation*}
 \sum_{i=1}^{\mu} t^{\alpha_i}=R(n,\beta_1,\ldots,\beta_g;t).
\end{equation*}
For almost all means for an open dense subset in the $\mu$-constant strata in a deformation space.  
\end{yano}

In 1989, B.~Lichtin~\cite{Li89} proved that for $i=1,\cdots,g$, 
the number $-\frac{r_i}{R_i}$ is a root of the Bernstein polynomial of $f$  with characteristic sequence
$(n, \beta_1,\beta_2,\ldots, \beta_g)$. These result  has been extended to the general curve case (not necessarily irreducible) by F.~Loeser
in \cite{Lo88}.

Yano's conjecture holds for $g=1$ as it was proved
by the second named author in~\cite{Pi88}. 

In \cite[Section~4.2]{MS:89} M.~Saito described
how can vary the Bernstein polynomial in $\mu$-constant deformations. Let $\{f_t\}_{t\in T}$  be a $\mu$-constant analytic deformation of an irreducible germ of an
isolated curve singularity $f_0$.  Then there exists an analytic stratification 
of $T$ (by restricting $T$ if necessary) such that the Bernstein polynomial is constant on each strata.
Since the $\mu$-constant strata is irreducible and smooth,  the Bernstein polynomial of its open stratum, 
denoted by $b_{\mu, \text{gen}}(s)$,  is called 
the Bernstein polynomial of the generic $\mu$-constant deformation of $f_0(x)$.

In this article we are interested in the case $g=2$.
Yano \cite{Y82} claimed the case $(4,6,2n-3)$, with $n\geq 5$, but referred to a non published article. 
For $g=2$, the characteristic sequence $(n,\beta_1,\beta_2)$ can be written as 
$(n_1n_2, mn_2, mn_2+q)$ where $n_1,m,n_2,q\in\mathbb{Z}_{>0}$ satisfying
$$
\gcd(n_1,m)=\gcd(n_2,q)=1.
$$
In this work we solve Yano's conjecture 
for the case
\begin{equation}\label{eq:simple_roots}
\gcd(q,n_1)=1\text{ or }\gcd(q,m)=1.
\end{equation}
The above condition is equivalent to ask for the algebraic monodromy to have distinct eigenvalues. 
In that case, the~$\mu$ $b$-exponents are all distinct and they coincide with
the opposite of roots of the reduced Bernstein polynomial (which turns out to be of degree~$\mu$).

Our goal is to compute the roots of the Bernstein polynomial
for a generic function having characteristic sequence $(n_1n_2, mn_2, mn_2+q)$.
 To do this we follow the same method than in \cite{Pi88}. 
To prove that a rational number is a root of the Bernstein polynomial of some function $f$, 
we prove that this number is a pole of some integral with a transcendental residue. 

For some exponents of the generating series we prove this property for 
families of functions which should contain \emph{generic} elements in the
$\mu$-constant stratum. 
For the rest of exponents, the computations are very tricky, 
and we apply them only to particular functions.
In order to ensure that the opposite of these exponents are
roots of the Bernstein polynomial for a generic $f$, we use the following result.

\begin{prop0}[{\cite[Corollary 21]{V80}}]\label{semi} Let $f_t(x)$  be a $\mu$-constant analytic deformation of  an  
isolated hypersurface singularity $f_0(x)$. 
If all  eigenvalues of  the monodromy are pairwise different, then all roots
of the reduced Bernstein-Sato polynomial $\tilde{b}_{f_t}(s)$ depend lower semi-continously upon the parameter $t$.
\end{prop0}

Then  if $\alpha$ is root of the local Bernstein-Sato polynomial $b_{f_0}(s)$ for some $f_0,$
and $\alpha+1$ is not root of $b_{f}(s)$  for any $f$ with the same characteristic sequence,
then by Proposition \ref{semi}, $\alpha$ is root of the local Bernstein-Sato polynomial $b_{f}(s)$
for $f$ generic with  the same characteristic sequence.

In the first section
 we collect some results on integrals that will be crucial in the following. 
Some of the proofs are in the appendix of the paper. In the second section  we express Yano's conjecture in our setting. In the third and fourth sections
we compute poles of integrals that we shall need later, and in the fifth part we show how we can use these integrals to compute roots of the Bernstein polynomial and we prove Yano's conjecture in the sixth section.

We are very grateful to  Driss Essouabry for providing us with Proposition~\ref{Essou1a}.

\numberwithin{equation}{section}

\section{Meromorphic integrals}\label{sec:mero-int}

\subsection{One-variable integrals}

Let $f\in \br [t]$ be a real polynomial such that $f(t)>0$ for all $t\in[0,1]$ and let 
 $a,b\in\bz,\ a\geq 0,\  b\geq 1$ fixed. Consider the (complex)  integral depending on a complex variable~$s\in \mathbb{C} $
\begin{equation}
\mathcal{Y}_{f,a,b}(s):=\int_0^1  f(t)^s t^{a s+b} \frac{dt}{t}.
\end{equation}
Using classical techniques we can see that this integral defines a holomorphic function
on a half-plane in $\mathbb{C}$ admitting a meromorphic continuation to the whole complex line,
having only simple poles at some rational numbers (with bounded denominator), where the residues can be controlled.

\begin{prop}\label{prop:int_una_var} The function $s\mapsto\mathcal{Y}_{f,a,b}(s)$ satisfies the following properties:

\begin{enumerate}
\enet{\rm(\arabic{enumi})} 
 \item  It is absolutely convergent for $\Re(s)>\alpha_0:=-\frac{b}{a}$ (the whole $\mathbb{C}$ if $a=0$).
\item  It has a meromorphic continuation on $\bc$ with simple poles, which are contained in
$S=\left\{ -\frac{b+k}{a}\mid k\in \mathbb{Z}_{\geq 0} \,\right\}$.
\item $\res_{s=-\frac{b+k}{a}}\mathcal{Y}_{f,a,b}(s)$ is algebraic over the field of coefficients of~$f$.
\end{enumerate}
\end{prop}

\begin{proof}
For the first statement, there exists $M_s>0$ such that 
$\left|f(t)^{s}\right|\leq M_s$ for $t\in [0,1]$.
Hence,
\begin{gather*}
\left|
\int_{0}^1 t^{as+b-1}f^s(t)dt\right|
\leq M_s\int_{0}^1 t^{a\Re(s)+b-1}dt
=M_s\left.\frac{t^{a\Re(s)+b}}{a\Re(s)+b}\right|_0^1=\frac{M_s}{a\Re(s)+b}.
\end{gather*}

For the second statement, we consider the Taylor expansion of $f(t)^s$ at $t=0$ of order~$k$:
$$
f^s(t)=\sum_{i=0}^k\frac{(f^s)^{(i)}(0)}{i!} t^{i}+t^{k+1} R_{s,k}(t),
\quad R_{s,k}(t)=\frac{1}{k!}\int_{0}^1 (1-u)^k (f^s)^{(k+1)}(u t)du.
$$
Hence,
$$
\mathcal{Y}_{f,a,b}(s)=\sum_{i=0}^k\frac{(f^s)^{(i)}(0)}{(a s+b+ i)i!}+ H(s)
$$
where
$$
H(s):=\int_0^{1} t^{a s +b+k}  R_{s,k}(t)dt. 
$$
Note that $H(s)$ is holomorphic for $\Re(s)>-\frac{ b+k+1}{a}$, and the first terms are rational functions.
Hence, the second statement is true.

For the third one, note that
$$
\res_{s=-\frac{b+k}{a}}\mathcal{Y}_{f,a,b}(s)=\frac{(f^{-\frac{b+k}{a}})^{(k)}(0)}{a k!}
$$ 
which satisfies the conditions.
\end{proof}

In general, we will deal with more general integrals which a priori, are not so well-defined.
For example, let $f(t),g(t)$
be two real analytic functions in $t^{\frac{1}{N}}$ in $[0,T]$, for some $N\in\mathbb{Z}_{>0}$ and $T>0$. Let $K$
be the field of coefficients of the power series of $f,g$ at~$0$. 
Let $r_f,r_g$ be the orders
of $f,g$ at~$0$, respectively, and assume that $f(t)>0$ for $t\in(0,T]$.
Let $a,b\in\mathbb{Q},\ a\geq 0,\  b> 0$ fixed. Consider the improper integral 
\begin{equation}
\mathcal{Y}_{f,g,a,b}(s):=\int_0^T  f(t)^s g(t) t^{a s+b} \frac{dt}{t}.
\end{equation}
Let us denote $a_1=a+r_f$ and $b_1=b+r_g$. The following result is a direct
consequence of the Proposition~\ref{prop:int_una_var}, using a simple change of variables.

\begin{cor}\label{cor:int_una_var} The function $s\mapsto\mathcal{Y}_{f,g,a,b}(s)$ satisfies the following properties:
\begin{enumerate}
\enet{\rm(\arabic{enumi})} 
 \item  It is absolutely convergent for $\Re(s)>\alpha_0:=-\frac{b_1}{a_1}$ (the whole $\mathbb{C}$ if $a_1=0$).
\item  It has a meromorphic continuation on $\bc$ with simple poles, which are contained in
$S=\left\{ -\frac{N b_1+k}{N a_1}\mid k\in \mathbb{Z}_{\geq 0} \,\right\}$.
\item $\displaystyle\res_{s=-\frac{N b_1+k}{N a_1}}\mathcal{Y}_{f,g,a,b}(s)$ is algebraic over~$K$.
\end{enumerate}
\end{cor}

\subsection{Two-variables integrals}

\begin{dfn}
We say that a real polynomial $f\in \br [x,y]$ is \emph{positive} if $f(x,y)>0$ for all 
$(x,y)\in[0,1]^2$. 
\end{dfn}
Let us state the two-variables counterpart of Proposition~\ref{prop:int_una_var}. Let $f\in \br [x,y]$ positive.
Let $a_1,a_2, b_1,b_2\in\bz$ such that $a_1,a_2\geq 0,  b_1,b_2 \geq 1 $. 
We denote
\begin{equation}
\mathcal{Y}(s)=\mathcal{Y}_{f,a_1,b_1,a_2,b_2}(s):=
\int_0^1  \int_0^1  f(x,y)^s x^{a_1s+b_1} y^{a_2s+b_2} \frac{dx}{x} \frac{dy}{y}.
\end{equation}

\begin{prop}[Essouabri]\label{Essou1a} The function $\mathcal{Y}(s)$ satisfies the following porperties:
\begin{enumerate}
\enet{\rm(\arabic{enumi})} 
 \item  It is absolutely convergent for $\Re(s)>\alpha_0$, where $\alpha_0=
\sup\left(-\frac{b_1}{a_1},-\frac{b_2}{a_2}\right)$
\item  It has a meromorphic continuation on $\bc$ with poles of order at most $2$ contained in 
$S=\left\{ -\frac{b_1+\nu_1}{a_1}, \, \, \nu_1 \in \mathbb{Z}_{\geq 0} \,\right\} \cup 
\left\{ -\frac{b_2+\nu_2}{a_2}, \, \, \nu_2\in \mathbb{Z}_{\geq 0} \,\right\}$
\end{enumerate}

\end{prop}

In order to do not break the line of the exposition, the proof of this Proposition
is given in the \ref{sec:appendix}. 
Note that no information is given in the above Proposition for the residues.
Let us introduce some notation.

\begin{ntc}
Let $f: [0,1]\to\br$ be a continous function. 
We will denote by $G_{f}(s)$ the meromorphic continuation of
$$
\int_0^1 f(t)t^s\frac{dt}{t}.
$$
\end{ntc}

\begin{prop}\label{continuation} 
With the hypotheses of Proposition{\rm~\ref{Essou1a}},
let $\nu_1\in \mathbb{Z}_{\geq 0}$ such that $\alpha=-\frac{b_1+\nu_1}{a_1}\neq -\frac{b_2+\nu_2}{a_2}$ 
for all $\nu_2\in \mathbb{Z}_{\geq 0}$, then the pole  of $\mathcal{Y}(s)$  at $\alpha$ is simple and 

\begin{equation}
\res_{s=\alpha} {\mathcal{Y}(s)}=\frac{1}{\nu_1! a_1} 
G_{ h_{\nu_1,\alpha,x}}(a_2\alpha+b_2),\quad
h_{\nu_1,\alpha,x}(y):=\frac{\partial^{\nu_1} f^\alpha}{\partial x^{\nu_1}}(0,y).
\end{equation}
\end{prop}

The proof of this Proposition is also given in the \ref{sec:appendix}.
Note that, under the hypotheses of the Proposition, the function 
$G_{ h_{\nu_1,\alpha,x}}$ admits an integral  expression which
is absolutely convergent and holomorphic
for  $\Re(s)> -N_2-1$,  with $N_2$ such that $\alpha>-  \frac{b_2+N_2+1}{a_2}$,
see the proof of Proposition~\ref{Essou1a} in page~\pageref{page-Essou1a}.
The following result is also a straightforward consequence of 
the proof of Proposition~\ref{Essou1a}.

\begin{prop}
Let $(\nu_1,\nu_2)\in \mathbb{Z}_{\geq 0}^2$ such that 
$\alpha=-\frac{b_1+\nu_1}{a_1}=-\frac{b_2+\nu_2}{a_2}$, then the pole at $\alpha$ is 
of order at most $2$ and 
$$
\lim_{s\to\alpha}\mathcal{Y}(s)(s-\alpha)^2= \frac{1}{\nu_1! \nu_2!a_1a_2}
\frac{\partial^{\nu_1+\nu_2} f^{\alpha}}{\partial x^{\nu_1}\partial y^{\nu_2}}(0,0).
$$
\end{prop}

We finish this section with a result that relates these integrals with the beta function.

\begin{lema}\label{beta}
Let $p\in \mathbb{N}$ and $c\in \mathbb{R}_{>0}$. 
Given $s_1, s_2\in \bc$ 
  such that  $-\alpha=s_1+s_2>0$
then 
\begin{equation}
G_{ \left( y^{p}+c\right)^\alpha}(p s_1)+
G_{ \left(1 +c x^{p}\right)^\alpha}(ps_2)=
\frac{c^{-s_2}}{p}\boldsymbol{B}\left( s_1, s_2\right)
\end{equation}
where $\boldsymbol{B}$ is the \emph{beta function}.
\end{lema}

The proof appears in the \ref{sec:appendix}.

\section{Candidate roots}

Since we are going to use mostly Bernstein polynomial
instead of $b$-exponents, it will be more convenient to
work with the opposite exponents. If we study closely the Yano's set of candidates for the $b$-exponents 
given by the exponents of the generating series~\eqref{eq:generating},
we can check that for
a branch with~$g$ characteristic pairs, this set can be decomposed in a union of $g$~subsets,
each one associated to a characteristic pair. For example, in the case $g=1$
and characteristic sequence $(n_1,m)$, with $\gcd(n_1,m)=1$, 
the set of opposite $b$-exponents  is decomposed into only one
set
\begin{equation} 
A:=\left\{-\frac{m+n_1+k}{mn_1}: 0\leq k< mn_1,\ 
\frac{m+n_1+k}{m},\frac{m+n_1+k}{n_1}\notin\bz\right\}.
\end{equation}
Note that $\max\ A=-\frac{m+n_1}{mn_1}$, which is the opposite of the  \emph{log canonical threshold} of the singularity
and we have
$$
\max A-1<\rho\leq\max A,\qquad\forall\rho\in A
$$
agreeing with \eqref{eq:saito}.
Recall that the conductor of the semigroup generated by $(n_1,m)$ is $m n_1-m-n_1$.

Let us consider the case $g=2$. 
Let us fix some notations.
We work with curve singularities with characteristic sequence
$(n_1 n_2, m n_2, m n_2+q)$, where 
\begin{itemize}
\item $1<n_1<m$, $\gcd(m,n_1)=1$;
\item $q>0$, $n_2>1$, $\gcd(q,n_2)=1$.
\end{itemize}
In order to use the integrals of \S\ref{sec:mero-int}, we will restrict to
\emph{real} singularities with Puiseux expansion
$$
x=\dots+a_1 y^{\frac{m}{n_1}}+\dots+a_2 y^{\frac{m n_2+q}{n_1 n_2}}+\dots,
$$
where $a_1,a_2\in\mathbb{R}^*$ (only characteristic terms are shown, the other coefficients are also real). 
The semigroup~$\Gamma$ of these singularities is
generated by $n_1 n_2$, $ m n_2$ and $ m n_1 n_2+q$. 
Its conductor equals
$$
n_2 (m n_{1} n_{2}+q) - (m+ n_{1}) n_{2}  - q + 1.
$$
We are going to deal 
with \emph{most} local irreducible curve singularities with two Puiseux pairs,
where \emph{most} stands for \emph{non-multiple eigenvalues for the algebraic monodromy}.
The condition on the eigenvalues is equivalent to \eqref{eq:simple_roots}.

\begin{ejm}
Let us consider $(a,b)\in\mathbb{Z}^2_{\geq 1}$ such that $m n_1 n_2+q=a m+b n_1$. 
Since the conductor of the semigroup generated by $n_1,m$ equals $(m-1)(n_1-1)$, we deduce that
such coefficients exist with the condition $a,b\geq 0$. We can prove that $a,b\geq 1$ using \eqref{eq:simple_roots}.
Then the functions 
$$
F_\pm(x,y)=(x^{n_1}\pm y^m)^{n_2}+x^a y^b
$$
define singularities of this type. 
\end{ejm}

Let us express Yano's set of opposite candidates as the union of two subsets $A_1,A_2$.
The first one looks like~$A$:
\begin{equation}\label{eq:A1} 
A_1:=\left\{\alpha=-\frac{m+n_1+k}{mn_1n_2}: 0\leq k< mn_1n_2, 
\text{ and }n_2 m\alpha,n_2 n_1 \alpha\notin\bz\right\};
\end{equation}
the last condition is equivalent to neither $m$ nor $n_1$ 
are divisors of $m+n_1+k$.
The second one corresponds to the second Puiseux pair:
\begin{equation}\label{eq:A2} 
A_2:=\left\{\alpha=-\frac{\overbrace{(m+n_1)n_2+q+k}^{N_k}}{n_2\underbrace{(mn_1n_2+q)}_{D}}\left\vert0\leq k<n_2D
\text{ and }n_2\alpha,D\alpha\notin\mathbb{Z}\right.\right\};
\end{equation}
the last condition is equivalent to neither $n_2$ nor $D$ 
are divisors of $N_k$.
They satisfy the following conditions:
\begin{enumerate}
\enet{(A\arabic{enumi})}
\item These two subsets are disjoint under the condition \eqref{eq:simple_roots}.
\item  $\max A_i-\min A_i<1$ for $i=1,2$
\item $-\max A_1$ is the \emph{log canonical threshold} of those singularities.
\item $0<\max A_1-\max A_2<1$.
\end{enumerate}

These subsets are decomposed as disjoint unions
$A_1=A_{11}\sqcup A_{12}$ and $A_2=A_{21}\sqcup A_{22}$
using the semigroups associated to the singularity.
The set $A_{11}$ is formed by the elements of $A_1$ whose numerator is in the semigroup
generated by $(m,n_1)$, i.e.,
\begin{equation}\label{eq:A11}
A_{11}:=\left\{\left.-\frac{m\beta_1+n_1\beta_2}{mn_1n_2}\in A_1\right|\beta_1,\beta_2\in\mathbb{Z}_{\geq 1}\right\}.
\end{equation}
The set $A_{21}$ is formed by the elements of $A_2$ whose numerator (minus~$q$) is in~$\Gamma$, i.e.,
\begin{equation}\label{eq:A21}
A_{21}:=\left\{\left.-\frac{N_k}{n_2 D}\right|N_k-q\in\Gamma\right\}.
\end{equation}
The following lemma means that $A_{12}$ and $A_{22}$ are somewhat \emph{small}.
\begin{lema}\label{lema:lct0}
If $\alpha\in A_{i2}$, $i=1,2$, then $\max A_1-\alpha<1$.
In an equivalent way
\begin{enumerate}
\enet{\rm(\arabic{enumi})}
\item 
if $-\frac{m+n_1+k}{mn_1n_2}\in A_{11}$, then $k\leq m n_1-m-n_1$;
\item 
if $-\frac{{N_k}}{n_2D}\in A_{21}$, then $\frac{{N_k}}{n_2D}<\frac{m+n_1}{m n_1 n_2}+1$. 
\end{enumerate}
\end{lema}

\begin{proof}
The first statement follows from the fact that $(m-1)(n_1-1)$
is the conductor of the semigroup generated by~$m,n_1$. 

For the second one, we use the conductor and $\Gamma$ to obtain
$$
N_k<n_2D-(m+n_1)n_2+ 1.
$$
Then,
\begin{equation*}
\frac{{N_k}}{n_2D}<1-\frac{(m+n_1)n_2- 1}{n_2D}<1+\frac{m+n_1}{m n_1 n_2}.
\end{equation*}

\vspace*{-5mm}
\end{proof}

\begin{remark}\label{spectrum}
 The connection between the set $\spec(f)$ of spectral numbers and roots of the Bernstein polynomial has been investigated by many authors. 
The spectral numbers  are such that $0<\tilde{\alpha}_1\leq \tilde{\alpha}_2 \leq \ldots \leq \tilde{\alpha}{ _\mu}$, where $\mu$ is the Milnor number. 
We know that  $\tilde{\alpha}_1=-\max A_1$ and
the set $\spec(f)$  is constant under $\mu$-constant deformation of the germ.
The main results in \cite{MS93,HS99,GH07}, imply that 
the set $\tilde{\alpha}\in \spec(f)$, such that $\tilde{\alpha}<\tilde{\alpha}_1+1$
are roots of the Bernstein polynomial $b_{f_t}(s)$ of every $\mu$-constant deformation 
$\{f_t\}$ of $f$. In fact, it can be proved that those spectral numbers are
contained in the set $A_{11} \cup A_{21}$ so a good chunk of the candidate roots are
already known to be roots of the Bernstein polynomial.
In a forthcoming paper~\cite{Bernstein-ACLM} the authors will describe the set of all common roots  of the Bernstein polynomial $b_{f_t}(s)$ of any $\mu$-constant deformation 
$\{f_t\}$ of $f$ with characteristic sequence $(n_1n_2, mn_2,mn_2+q)$ such that $\gcd(q,n_1)=1$ or $\gcd(q,m)=1$.
\end{remark}

\section{\texorpdfstring{Residues of integrals at poles in $A_1$}{Residues of integrals at poles in A1}}
\label{sec:res1}

\begin{dfn} \hspace*{-2mm} A polynomial $f\in\br[x,y]$
is called to be of \emph{type $(n_1n_2, mn_2,mn_2+q )^+$ }
if it satisfies:
\begin{equation}\label{eq:f_gen1}
f(x,y)=(x^{n_1}+ y^{m}+h_1(x,y))^{n_2}+ x^a y^b+h_2(x,y)
\end{equation}
where
\begin{enumerate}[series=G]
\enet{(G$^+$\arabic{enumi})}
\item\label{G1+} $ h_1(x,y)=\sum_{(i,j)\in\mathcal{P}_{n_1,m}} a_{i j} x^i y^j\in\mathbb{R}[x,y]$,
where
$$
\mathcal{P}_{n_1,m}:=\{(i,j)\in\mathbb{Z}_{\geq 0}^2\mid m i+n_1 j> m n_1\};
$$
\item $a,b\geq 0$ such that $a m+b n_1=mn_1n_2+q$;
\item the polynomial $h_2\in\mathbb{R}[x,y]$, whose support is disjoint from the first term,
satisfies that the characteristic sequence of $f$ is $(n_1n_2, mn_2,mn_2+q )$;
\item $f>0$ in $[0,1]^2\setminus\{(0,0)\}$.
\end{enumerate}
\end{dfn}

For $\beta_1,\beta_2\in\bz_{\geq 1}$, and $f$ of type $(n_1n_2, mn_2,mn_2+q )^+$  we set:
\begin{equation}
I(f,\beta_1,\beta_2)(s)=\int_0^1  \int_0^1  f(x,y)^s\, x^{\beta_1} 
y^{ \beta_2}\, \frac{dx}{x}\, \frac{dy}{y}.
\end{equation}
Note that $f$ does not satisfy the conditions stated in~\S\ref{sec:mero-int}
and we cannot ensure that $I(f,\beta_1,\beta_2)(s)$ is well-defined,
because $f(0,0)=0$. The purpose
of the following Proposition is to prove that, after a suitable change of variables, $I(f,\beta_1,\beta_2)(s)$
is expressed as  a linear combination of integrals as in Proposition~\ref{Essou1a}.
In order to simplify the notation, we denote  $\tilde{h}_2(x,y):=x^a y^b+h_2(x,y)$. We will use the following
properties:
\begin{enumerate}[resume*=G]
\enet{(G$^+$\arabic{enumi})}
\item The minimum degree of $h_1(x^m,y^{n_1})$ is greater than $m n_1$.
\item The minimum degree of $\tilde{h}_2(x^m,y^{n_1})$ is greater than $m n_1 n_2$.
\end{enumerate}

\begin{prop}\label{prop:newton1}
Let $f$ be of type $(n_1n_2, mn_2,mn_2+q )^+$
and $\beta_1,\beta_2\in\bz_{\geq 1}$. The integral $I(f,\beta_1,\beta_2)(s)$ is  absolutely 
convergent for $\Re{(s)}>-\frac{\beta_1m+\beta_2n_1}{mn_1n_2}$
and may have simple poles only for $s=-\frac{\beta_1m+\beta_2n_1+\nu}{mn_1n_2}, \nu\in \bz_{\geq 0}$.
\end{prop}

\begin{proof}
In this proof we are going to transform $I(f,\beta_1,\beta_2)(s)$ in a sum
of integrals of type $\mathcal{Y}(s)$, for which we may apply
Proposition~\ref{Essou1a}.
For the first step, we apply the change of variables
$$
x=x_1^{m},\quad y=y_1^{n_1}.
$$
Let us denote
$$
\tilde{f}(x_1,y_1):=f(x_1^{m},y_1^{n_1})= 
(x_1^{mn_1}+ y_1^{mn_1}+h_1(x_1^{m},y_1^{n_1}))^{n_2}+\tilde{h}_2(x_1^{m},y_1^{n_1}).
$$
We obtain (after renaming back the coordinates to $x,y$):
\begin{equation*}
I(f,\beta_1,\beta_2)(s)=mn_1\int_0^1\!\!\!  \int_0^1  \tilde{f}(x,y)^s\, x^{m \beta_1} 
y^{n_1 \beta_2}\, \frac{dx}{x}\, \frac{dy}{y}.
\end{equation*}
Let us decompose the square $[0,1]^2$ into two triangles
$$
D_1:=\{(x,y)\in[0,1]^2\mid x\geq y\},\quad D_2:=\{(x,y)\in[0,1]^2\mid x\leq y\}.
$$
We express 
\begin{equation}
I(f,\beta_1,\beta_2)(s)=m n_1(I_1(f,\beta_1,\beta_2)(s)+I_2(f,\beta_1,\beta_2)(s))
\end{equation}
 where 
each integral $I_j$ has as integration domain $D_j$:
$$
I_1(f,\beta_1,\beta_2)(s)=
\int_0^1\left(\int_0^x \tilde{f}(x,y)^s\,  
y^{n_1 \beta_2}\,  \frac{dy}{y}\right)x^{m \beta_1}\frac{dx}{x}
$$
and 
$$
I_2(f,\beta_1,\beta_2)(s)
=\int_0^1\left(
\int_0^y \tilde{f}(x,y)^s\, x^{m \beta_1} 
\, \frac{dx}{x}\right)y^{n_1 \beta_2}\, \frac{dy}{y}.
$$
Let us study first $I_1(f,\beta_1,\beta_2)(s)$.
We consider the change of variables 
$$
x=x_1,\quad y= x_1 y_1.
$$
There is a polynomial $f_{1}(x_1,y_1)$ determined by $\tilde{f}(x_1,x_1 y_1)=x_1^{m n_1 n_2}f_{1}(x_1,y_1)$.
Renaming the variables,
$$
f_{1}(x,y)=(1+y^{m n_1}+x h_{11}(x,y))^{n_2}+x \tilde{h}_{21}(x,y),\quad h_{11},\tilde{h}_{21}\in\br[x,y].
$$ 
The integral becomes
\begin{equation}\label{primera-integral}
I_1(f,\beta_1,\beta_2)(s)=\int_0^1  \int_0^1  f_{1}(x,y)^s\, 
x^{m \beta_1+n_1 \beta_2+mn_1n_2 s} 
y^{n_1 \beta_2}\, \frac{dx}{x}\, \frac{dy}{y}.
\end{equation}

We study now $I_2(f,\beta_1,\beta_2)(s)$ with the change of variables
$$
x=x_1 y_1,\quad y= y_1.
$$
As above,
there is a polynomial $f_{2}(x_1,y_1)$ such that $\tilde{f}(x_1 y_1,y_1)=y_1^{m n_1 n_2}f_{2}(x_1,y_1)$.
Renaming the variables,
$$
f_{2}(x,y)=(x^{mn_1}+1+yh_{12}(x,y))^{n_2}+y \tilde{h}_{22}(x,y),\quad h_{12},\tilde{h}_{22}\in\br[x,y].
$$ 
The integral becomes
\begin{equation}\label{primera-integral-bis}
I_2(f,\beta_1,\beta_2)(s)=\int_0^1  \int_0^1  f_2(x,y)^s\, 
x^{m \beta_1} 
y^{m \beta_1+n_1 \beta_2+mn_1n_2 s}\, \frac{dx}{x}\, \frac{dy}{y}.
\end{equation}
The key point is that the functions $f_1(x,y)$ and $f_2(x,y)$ are \emph{positive}, i.e.,
they do not vanish at $(0,0)$
and we can apply Proposition~\ref{Essou1a}. 
Therefore $I_1(f,\beta_1,\beta_2)(s)$ and  $I_2(f,\beta_1,\beta_2)(s)$
are absolutely convergent for $\Re(s)>-\frac{m \beta_1+n_1 \beta_2}{mn_1n_2}$,
have meromorphic continuation to the whole  plane $\bc$ with possible simple poles  at 
$\alpha=-\frac{m \beta_1+n_1 \beta_2+\nu}{mn_1n_2}$ with $\nu \in \bz_{\geq 0}$.
\end{proof}

We study the possible poles $\alpha\in A_1$, defined in \eqref{eq:A1}.

\subsection{\texorpdfstring{Residues at poles in $A_{11}$}{Residues at poles in A11}}
\mbox{}

In this subsection, let $\alpha \in A_{11}$, i.e. there exist $\beta_1,\beta_2\in\bz_{\geq 1}$ for which
\begin{equation}\label{eq:alphaA11}
\alpha=-\frac{m\beta_1+n_1\beta_2}{mn_1n_2},
\end{equation}
see \eqref{eq:A11}.

\begin{prop}\label{prop:a11}
Let $f$ be of type $(n_1n_2, mn_2,mn_2+q )^+$. Then, the integral
$I(f,\beta_1,\beta_2)(s)$ has a pole for $s=\alpha$ and its residue
is  $\frac{1}{mn_1n_2}\boldsymbol{B}\left(\frac{\beta_1}{n_1},\frac{\beta_2}{m}\right)$.
\end{prop}

\begin{proof}
With the notation in the proof of Proposition~\ref{prop:newton1},
one has 
$$
f_1^\alpha (0,y)=
(1+ y^{mn_1})^{n_2
\alpha},\quad f_2^\alpha (x,0)=
(x^{mn_1}+1)^{n_2
\alpha}.
$$
The residues of the integrals $I_1,I_2$ are computed using Proposition~\ref{continuation}.
For $I_1$, we have $(a_1,b_1)=(m n_1 n_2,m\beta_1+n_1\beta_2)$ and $(a_2,b_2)=(0,n_2\beta_2)$:
$$ 
 \res_{s=\alpha} I_1(f,\beta_1,\beta_2)(s)=\frac{1}{m n_1 n_2}
G_{ f_1^\alpha (0,\cdot)}(n_1\beta_2).
$$
With the same ideas,
$$
 \res_{s=\alpha} I_2(f,\beta_1,\beta_2)(s)=\frac{1}{m n_1 n_2}
G_{ f_2^\alpha (\cdot,0)}(m\beta_1).
$$ 
Recall that $I=m n_1(I_1+I_2)$. We apply Lemma~\ref{beta}
where $c=1$, $p=m n_1$, $\alpha=n_2\alpha$, $s_1=\frac{\beta_1}{n_1}$ and $s_2=\frac{\beta_2}{m}$,
and we obtain
\begin{equation*}
  \res_{s=\alpha} I(f,\beta_1,\beta_2)(s)=\frac{1}{mn_1n_2}
\boldsymbol{B}\left(\frac{\beta_1}{n_1},\frac{\beta_2}{m}\right).
\end{equation*}

\vspace*{-5mm}
\end{proof}

\begin{remark}\label{obs:sch}
Let $\alpha \in A_{11}$. Since $A_{11}\subset A_1$, the rational number $-n_2\alpha$ is not
an integer by~\eqref{eq:A1}. From the definition of $\alpha$ in \eqref{eq:alphaA11},
it is clear that if $\frac{\beta_1}{n_1}\in\mathbb{Z}$, then $m n_2\alpha\in\mathbb{Z}$
also in contradiction with~\eqref{eq:A1}. Hence 
$\frac{\beta_1}{n_1},\frac{\beta_2}{m}$ are not integers. 
Then, using a Theorem of Schneider in~\cite{sch}, 
we know that $\boldsymbol{B}\left(\frac{\beta_1}{n_1},\frac{\beta_2}{m}\right)$ is transcendental. 
\end{remark}

\subsection{\texorpdfstring{Residues at poles in $A_{12}$}{Residues at poles in A12}}\label{subsec:A12}
\mbox{}

In the above subsection, we have succeeded to compute the exact residue because in the application
of Proposition~\ref{continuation}, no derivation was needed. For elements in $A_{12}$ 
the situation is much more complicated
and we will restrict our computation to some particular examples.
Let us fix $\alpha=-\frac{m+n_1+k}{mn_1n_2} \in A_{12}$. We can 
express
\begin{equation}\label{eq:ij}
m i_0+n_1 j_0=mn_1+k\text{ for some }(i_0,j_0)\in\mathbb{Z}_{\geq 0}^2,
\end{equation}
since $m n_1$ is greater
than the conductor of the semigroup generated by $m,n_1$. 
Let
$$
f_{+t}(x,y):=(x^{n_1}+y^m+t x^{i_0} y^{j_0})^{n_2}+x^a y^b,\quad t\in\mathbb{R}_{>0},
$$
with $a$ and $b$ as in~(\ref{eq:f_gen1}).

\begin{prop}\label{prop:poles-a12}
The function
$I(f_{+t},1,1)(s)$ has a pole for $s=-\alpha$ 
and its residue is a polynomial of degree~$1$ in~$t$  whose
coefficient of~$t$ equals
$$
\frac{\alpha}{n_2n_1m} 
\boldsymbol{B}\left(\frac{1+ i_0}{n_1},\frac{1+j_0}{m}\right).
$$
\end{prop}

\begin{proof}
From Lemma~\ref{lema:lct0}, $1\leq k\leq mn_1-m-n_1$.
The computation of the residue of $I_1(f,1,1)(s)$ is quite involved for a general
polynomial and this is why we restrict our attention to $f_{+t}$. 
In the notation of Proposition~\ref{prop:newton1}, we have
$$
\tilde{f}_{+t}(x,y)=(x^{m n_1}+y^{m n_1}+t x^{i_0 m} y^{j_0 n_1})^{n_2}+x^{a m} y^{b n_1}.
$$
Then
$$
f_1(x,y)\!=\!(1+y^{m n_1}+t x^{k} y^{j_0 n_1})^{n_2}+x^{q} y^{b n_1},\
f_2(x,y)\!=\!(x^{m n_1}+1+t x^{i_0 m} y^{k})^{n_2}+x^{a m} y^{q}.
$$
By Proposition~\ref{continuation}, we have:
\begin{equation}\label{kpolos1}
\res_{s=\alpha} I_1(f_{+t},1,1)(s)=\frac{1}{m n_1 n_2 k!}
G_{ h_{k,\alpha,x}}(n_1),\quad
h_{k,\alpha,x}(y)=\frac{\partial^k f_1^\alpha}{\partial x^k}(0,y).
\end{equation}
It is well-known that
\begin{equation}\label{eq:derivar}
\frac{\partial^k f_1^\alpha}{\partial x^k}=
\alpha f_1^{\alpha-1}\frac{\partial^k f_1}{\partial x^k}+
\text{terms involving }f_1^{\alpha-m}\text{ and }\frac{\partial^{r} f_1}{\partial x^{r}}\text{ with }r<k.
\end{equation}
In the sequel  $\dots$ will mean in this proof \emph{independent of the variable}~$t$.
It is easy to obtain the coefficient of~$t$ (e.g., derivating with respect to~$t$ and replacing $t$ by~$0$):
$$
\frac{\partial^r f_1}{\partial x^r}(0,y)=
\begin{cases}
\dots&\text{ if }r<k,\\
tk! y^{n_1j_0}(1+y^{n_1m})^{n_2-1}+\dots&\text{ if }r=k.
\end{cases}
$$
Thus
$$
\frac{\partial^k f_1^\alpha}{\partial x^k}(0,y) =tk!\alpha y^{n_1j_0 }(1+y^{n_1m})^{n_2\alpha-1}+\dots
$$
The same arguments yield
$$
\res_{s=\alpha} I_2(f_{+t},1,1)(s)=
\frac{1}{m n_1 n_2 k!}
G_{ h_{k,\alpha,y}}(n_1),\quad
h_{k,\alpha,y}(x)=\frac{\partial^k f_2^\alpha}{\partial y^k}(x,0)
$$
and
$$
\frac{\partial^k f_2^\alpha}{\partial y^k}(x,0)=tk! \alpha x^{mi_0 }(x^{n_1m}+c)^{n_2\alpha-1}+\dots.
$$
Hence
$$
\res_{s=\alpha}I_1(f_{ +t},1,1)(s)=t\frac{\alpha}{m n_1 n_2} G_{(1+y^{n_1m})^{-n_2\alpha-1}}(n_1(j_0+1)) +\dots
$$
and
$$
\res_{s=\alpha}I_2(f_{ +t},1,1)(s)=t\frac{\alpha}{m n_1 n_2} G_{(x^{n_1m}+1)^{-n_2\alpha-1}}(m(i_0+1)) +\dots
$$
If we apply Lemma~\ref{beta} to $\alpha=-n_2\alpha-1$, $s_1=\frac{i_0+1}{n_1}$, $s_2=\frac{j_0+1}{m}$, $p=n_1m$,
we obtain
\begin{equation*}
\res_{s=\alpha} I(f_{ +t},1,1)(s)=t\frac{\alpha}{n_2n_1m} 
\boldsymbol{B}\left(\frac{1+ i_0}{n_1},\frac{1+j_0}{m}\right)+\dots
\end{equation*}
\end{proof}

\begin{remark}
Since $\alpha\in A_{12}\subset A_1$, by \eqref{eq:A1}, it is clear that $-n_2\alpha-1\notin\mathbb{Z}$, and this
number is the sum of the arguments of~$\boldsymbol{B}$. 
If $\frac{i_0+1}{n_1}\in\mathbb{Z}$, then $n_1$ divides $m+k$ and this is forbidden by~\eqref{eq:A1}.
Hence $\frac{i_0+1}{n_1},\frac{j_0+1}{m}\notin\mathbb{Z}$. Since these three rational numbers
are non-integers, we deduce from~\cite{sch} that $\boldsymbol{B}\left(\frac{1+ i_0}{n_1},\frac{1+j_0}{m}\right)$
is transcendental.
\end{remark}

\section{\texorpdfstring{Residues of integrals at poles in $A_2$}{Residues of integrals at poles in A2}}
\label{sec:res2}
 \begin{dfn} 
A polynomial $f\in \br[x,y]$
is  said of type $(n_1 n_2, mn_2,mn_2+q )^-$ 
if it satisfies:
\begin{equation}\label{eq:f_gen2}
f(x,y)=g(x,y)^{n_2}+ x^a y^b+h_2(x,y)
\end{equation}
where 
$g(x,y):=x^{n_1}-y^{m}+h_1(x,y)$
\begin{enumerate}
\enet{(H\arabic{enumi})}
\item $h_1(x,y)$ is as in~\ref{G1+}.
\item $a,b\geq 0$ such that $a m+b n_1=mn_1n_2+q$.
\item\label{H3} There exists $a_1,\dots,a_k\in\br$ such that for
\begin{equation*}
Y(x^{\frac{1}{m}}):=\left(x^{\frac{1}{m}}+a_1x^{\frac{2}{m}}+\cdots+a_k x^{\frac{k+1}{m}}\right)^{n_1}
\end{equation*}
we have 
$\ord_x g(x,Y(x^{\frac{1}{m}}))>\frac{mn_1n_2+q}{mn_2}$ and $Y(x^{\frac{1}{m}})>0$ if $0<x\leq 1$.
Let
$$
g_Y(x,y):=\prod_{\zeta_m^m=1}\left(y-Y(\zeta_m x^{\frac{1}{m}})\right)\in\br[x,y].
$$
\item The polynomial $h_2\in\mathbb{R}[x,y]$, whose  support  is disjoint from the first terms, 
satisfies  that the characteristic sequence of $f$ is $(n_1n_2, mn_2,mn_2+q )$.
\item Let $\mathcal{D}_{Y}:= \{(x,y)\in\br^2\mid 0\leq x\leq 1, 0\leq y\leq Y(x^{\frac{1}{m}})\}$. 
Then $f> 0$ on $\mathcal{D}_{Y} \setminus \{(0,0)\}$.
\end{enumerate}
\end{dfn}

For $\beta_1,\beta_2\in\bz_{\geq 1}$,  $ \beta_3\in \mathbb{Z}_{\geq 0}$ and $f$ of type $(n_1n_2, mn_2,mn_2+q )^-$
(with $g,Y$ as above)  we set:
\begin{equation}
\mathcal{I}(f,\beta_1,\beta_2,\beta_3)(s)=\int\!\!\!\int_{\mathcal{D}_{Y}}  f(x,y)^s\, x^{\beta_1} 
y^{ \beta_2}g_Y(x,y)^{\beta_3}\, \frac{dx}{x}\, \frac{dy}{y}.
\end{equation}

\begin{prop}\label{prop:newton2}
Let $f\in\br[x,y]$ be a polynomial of type $(n_1 n_2, mn_2,mn_2+q )^-$,
$\beta_1,\beta_2\in\mathbb{Z}_{\geq 1}$ and 
$\beta_3\in\mathbb{Z}_{\geq 0}$.
Then the integral $\mathcal{I}(f,\beta_1,\beta_2,\beta_3)(s)$ is convergent for 
$\Re{(s)}>-\frac{\beta_1m+\beta_2n_1+\beta_3mn_1}{mn_1n_2}$ and
its set of poles is contained in the set 
$$
P_{1}\cup\bigcup_{i\in\mathbb{Z}_{\geq 1},j\in\mathbb{Z}_{\geq 0} } P_{2,i,j}
$$
where
$$
P_{1}:=\left\{\left.-\frac{m\beta_1+n_1\beta_2+m n_1\beta_ 3+\nu}{mn_1n_2}\right|\nu\in\mathbb{Z}_{\geq0}\right\}
$$
and
$$
P_{2,i,j}:=\left\{\left.-\frac{n_2(m \beta_1+n_1 \beta_2+m n_1\beta_3+j)+q(\beta_3+i)+\nu}{n_2(mn_1n_2+q)}\right|\nu\in\mathbb{Z}_{\geq0}\right\}
$$
The poles have at most order two. The poles may have order two at the values contained in  
$P_{1} $ and $P_{2,i,j}$ for some $i,j$.
\end{prop}

\begin{proof}
We proceed as in the proof of Proposition~\ref{prop:newton1}. 
We start with the change: $x=x_1^m$, $y=y_1^{n_1}$. Note that after this change, the integration domain
is exactly
$$
\mathcal{D}_1:=\{(x,y)\in\mathbb{R}^{2},0\leq x\leq 1,0\leq y\leq Y_1(x)\},
$$ 
where $Y_1(x)=Y(x)^{\frac{1}{n_1}}=x+a_1x^2+\cdots+a_kx^{k+1}$.
We rename
the coordinates and we obtain,
\begin{equation*}
\mathcal{I}(f,\beta_1,\beta_2,\beta_3)(s)=mn_1\int\!\!\!\int_{\mathcal{D}_1}f(x^m,y^{n_1})^s\, 
x^{m\beta_1} 
y^{n_1 \beta_2} g_{0,Y}(x,y)^{\beta_3}\, \frac{dx}{x}\, \frac{dy}{y}.
\end{equation*}
where $g_0(x,y):=g(x^{m},y^{n_1})$ and $\ord g_0(x,Y_1(x))>\frac{m n_1 n_2+q}{n_2}$ and 
$g_{0,Y}(x,y)$ is defined in the same
way and satisfies $g_{0,Y}(x,Y_1(x))\equiv 0$.

The following change is $x=x_1$, $y=x_1 y_1$. 
Let $\tilde{g}(x,y)$ be defined such that $g_0(x_1,x_1 y_1)=x_1^{n_1 m}\tilde{g}(x_1,y_1)$.
Let $Y_2(x)=\frac{Y_1(x)}{x}$, note that $\ord \tilde{g}(x,y)>\frac{q}{n_2}$. In the same
way, we define $\tilde{f}(x,y)$ such that $f(x_1^{m},x_1^{n_1} y_1^{n_1})=x_1^{n_1 m n_2}\tilde{f}(x_1,y_1)$.
It is easily seen that
$$
\tilde{g}(x,y)\!=\!1-y^{n_1 m}+x^{-n_1 m}h_1(x^m,x^{n_1} y^{n_1}),\
\tilde{f}(x,y)\!=\!\tilde{g}(x,y)^{n_2}+x^{q} 
y^{n_1b}+ \tilde{h}_2(x,y+1),
$$
where the Newton polygon of $\tilde{h}_2(x,y)$ is above the one of $y^{n_2}+x^q$
(from the condition of~$f$ having the chosen characteristic sequence). 
We define $g_{0,Y}(x_1,x_1 y_1):=x_1^{n_1 m}\tilde{g}_Y(x,y)$ in the same way
and ${\tilde{g}}_Y(x,Y_2(x))\equiv 0$.

Let
$$
\mathcal{D}_2=\{(x,y)\in\mathbb{R}^{2},0\leq x\leq 1,0\leq y\leq Y_2(x)\}.
$$
With the renaming of coordinates, we have
\begin{equation}
\mathcal{I}(f,\beta_1,\beta_2,\beta_3)(s)=mn_1\int\!\!\!\int_{\mathcal{D}_2}\tilde{f}(x,y)^s\, 
x^{M+mn_1n_2 s} 
y^{n_1 \beta_2} \tilde{g}_Y(x,y)^{\beta_3}\, \frac{dx}{x}\, \frac{dy}{y},
\end{equation}
where $M:=m \beta_1+n_1 \beta_2+m n_1\beta_3$.

Note that $\tilde{f}$  is strictly positive on $\mathcal{D}_2\setminus \{x=0\}$ and $\tilde{f}(0,y)=1-y^{mn_1}$. 
Then $\tilde{f}>0 $ on $\mathcal{D}_2\setminus \{(0,1)\}$. This is why we perform the change
of variables $x=x_1,y=(1-y_1)Y_2(x_1)$. From the above properties if
$\hat{g}(x,y)=\tilde{g}(x,(1-y)Y_2(x))$, its Newton polygon is \emph{more horizontal} 
than the one of $y^{n_2}+x^{q}$ and
the coefficient of~$y$ equals $m n_1$.
In particular, if $\hat{f}(x,y)=\tilde{f}(x,(1-y)Y_2(x))$, then
\begin{equation*}
\hat{f}(x,y)=(m n_1 y)^{n_2}+x^q+\hat{h}(x,y)
\end{equation*}
where the Newton polygon of $\hat{h}(x,y)$ is above the one of the first two monomials.

Since ${\tilde{g}}_Y(x,Y_2(x))\equiv 0$ then 
$$
{\tilde{g}}(x,(1-y)Y_2(x))=yq_Y(x,y),\quad q_Y(0,0)=-n_1.
$$
Let us define $\hat{g}_Y(x,y)$ by
$$
y^{\beta_3}\hat{g}_Y(x,y)={\tilde{g}}_Y(x,(1-y)Y_2(x))^{\beta_3}((1-y)Y_2(x))^{n_1\beta_2-1},\
\hat{g}_Y(x,y)=\sum b_{ij}x^jy^{i-1}.
$$
This change of variables transforms the integration domain~$\mathcal{D}_2$ into the
square $[0,1]^2$.
Then,
\begin{equation}\label{primera-integral-3polos}
\mathcal{I}(f,\beta_1,\beta_2,\beta_3)(s)\!=\!mn_1
\int_0^1\!\!\!\!  \int_0^1\!\!\!  
\hat{f}(x,y)^s\, 
x^{M+mn_1n_2 s} 
y^{\beta_3+1}\hat{g}_Y(x,y)\, \frac{dx}{x}\, \frac{dy}{y},
\end{equation}
where $\hat{g}_Y(x,y)\in\br[x,y]$.

We break this integral as
\begin{equation}\label{eq:I_descomposicion}
\mathcal{I}(f,\beta_1,\beta_2,\beta_3)(s)=mn_1
\sum_{i\geq1,j\geq 0} b_{i,j} J_{i,j}(s),\quad b_{1,0}=1,
\end{equation}
where 
$$
J_{i,j}(s):=\int_0^1\!\!\!\int_0^1  \hat{f}(x,y)^s
x^{M+j+mn_1n_2 s} 
y^{\beta_3+i} \frac{dx}{x}\, \frac{dy}{y}.
$$
Each of these integrals looks like the ones in Proposition~\ref{prop:newton1} and we apply the same procedure
where $(n_1,m)$ is replaced by $(q,n_2)$. Hence, we get $J_{i,j}(s)=J_{i,j,1}(s)+J_{i,j,2}(s)$.
Replacing $\beta_1$ by $M+j+mn_1n_2 s$ and $\beta_2$ by $\beta_3+i$ in the statement
of Proposition~\ref{prop:newton1}, we obtain 
\begin{equation}\label{eq:ji1}
J_{i,j,1}(s)=n_2q \int_0^1 \!\!\! \int_0^1  F_1(x,y)^s\, 
x^{s n_2 D+B_{i,j}} 
y^{q(\beta_3+i)}\, \frac{dx}{x}\, \frac{dy}{y},
\end{equation}
where $B_{i,j}=n_2 (M+j)+q(\beta_3+i)$ and $D=m n_1n_2+q$ as in \eqref{eq:A2}, 
and
\begin{equation}\label{eq:ji2} 
J_{i,j,2}(s)=n_2q \int_0^1 \!\!\! \int_0^1  F_2(x,y)^s\, 
x^{n_2(M+mn_1n_2 s+j)} 
y^{s n_2 D+B_{i,j}}
\, \frac{dx}{x}\, \frac{dy}{y},
\end{equation}
where
$F_1,F_2$ are strictly positive in the square.
The poles of $J_{i,j,1}(s)$ are simple and given by
$$
\alpha=-\frac{n_2(m \beta_1+n_1 \beta_2+m n_1\beta_3+j)+q(\beta_3+i)+\nu}{n_2(mn_1n_2+q)},
\quad \nu\in\mathbb{Z}_{\geq 0}.
$$
The poles of $J_{i,j,2}(s)$ are the above ones and 
$$
\alpha =-\frac{m\beta_1+n_1\beta_2+m n_1\beta_ 3+j+\nu}{mn_1n_2},\quad \nu\in\mathbb{Z}_{\geq 0};
$$
they may be double if one element is of both types (for fixed~$i,j,\beta_1,\beta_2,\beta_3$).
\end{proof}

\subsection{\texorpdfstring{Residues at poles in $A_{21}$}{Residues at poles in A21}}
\mbox{}

Let $\alpha\in A_{21}$. Because of the definition~\eqref{eq:A21} of $A_{21}$ and the structure
of the semigroup $\Gamma$, there exist $\beta_1,\beta_2\in\mathbb{Z}_{\geq 1}$
and $\beta_3\in\mathbb{Z}_{\geq 0}$
such that
\begin{equation}\label{eq:alpha_A21}
\alpha=-\frac{n_2(\beta_1 m+\beta_2 n_1)+\beta_3(mn_1n_2+q)+q}{n_2 (mn_1n_2+ q)}.
\end{equation}

\begin{prop}\label{prop:a21}
For any $f$ of type $(n_1 n_2 , m n_2 , m n_2 +q)^-$,
$\alpha$ is a pole of the integral
$\mathcal{I}(f,\beta_1,\beta_2,\beta_3)(s)$ with residue
$$
\frac{1}{n_2(mn_1n_2+q)} 
\boldsymbol{B}\left(\frac{\beta_3+1}{n_2},-\alpha-\frac{\beta_3+1}{n_2}\right)
$$
\end{prop}

\begin{proof}
We keep the notations of Proposition~\ref{prop:newton2}. If $i>1$ or $j>0$ then 
$$
 \res_{s=\alpha} J_{i,j,1}(s)=\res_{s=\alpha} J_{i,j,2}(s)= 0,
$$
since the starting point of the poles is shifted by~$1$ to the left and $\alpha$
is in the semiplane of holomorphy.

We compute the residues for $J_{1,0,1}(s)$ and $J_{1,0,2}(s)$ using Proposition~\ref{continuation}.
Using \eqref{eq:ji1}, we have $\nu_1=0$, $a_1=n_2 (mn_1n_2+ q)$,
$b_1=n_2(\beta_1 m+\beta_2 n_1)+\beta_3(n_2 m_1 n_1+q)+q$,
$a_2=0$, $b_2=q(\beta_3+1)$; hence
$$
\res_{s=\alpha} J_{1,0,1}(s)= \frac{q}{mn_1n_2+q} 
G_{\left( (mn_1)^{n_2}y^{n_2q}+1 \right)^\alpha}(q(\beta_3+1)).
$$
We apply the same computations (the roles of $x$ and $y$ exchange), where
now $a_2=m n_1 n_2 ^2$, $b_2=n_2(\beta_1 m+\beta_2 n_1+\beta_3 m_1 n_1)$. Hence,
$$
\res_{s=\alpha} J_{1,0,2}(s)= \frac{q}{mn_1n_2+q} 
G_{\left((mn_1)^{n_2}+x^{n_2q} \right)^\alpha}(n_2(m\beta_1+n_1\beta_2+mn_1\beta_3+mn_1n_2\alpha)).
$$
Let us apply Lemma~\ref{beta} ($x,y$ are exchanged). We have $\alpha=\alpha$, $s_2=\frac{\beta_3+1}{n_2}$,
$s_1=\frac{m\beta_1+n_1\beta_2+mn_1\beta_3+mn_1n_2\alpha}{q}$, $p=n_2 q$ and $c=(m n_1)^{n_2}$.
The condition is fullfilled:
\begin{gather*}
s_2\!+\!s_1\!\!=\!\frac{\beta_3+1}{n_2}\!\!+\!\!\frac{m\beta_1+n_1\beta_2+mn_1\beta_3}{q}\!-\!
m n_1\frac{n_2(\beta_1 m\!+\!\beta_2 n_1)\!+\!\beta_3(n_2 m_1 n_1\!+\!q)\!+\!q}{q (m n_1 n_2+q)}\\
=\frac{\beta_3+1}{n_2}+ 
\frac{m\beta_1+n_1\beta_2}{q}\left(1-\frac{m n_1 n_2}{m n_1 n_2+q}\right)
-\frac{m n_1 }{m n_1 n_2+q}=\\
\frac{\beta_3+1}{n_2}+\frac{m\beta_1+n_1\beta_2}{m n_1 n_2+q}
-\frac{m n_1 }{m n_1 n_2+q}=-\alpha.
\end{gather*}
Hence,
 $$
 \res_{s=\alpha} (J_{1,0,1}(s)+J_{1,0,2}(s))= \frac{1}{(mn_1)^{\beta_3+1}n_2(mn_1n_2+q)} 
\boldsymbol{B}\left(\frac{\beta_3+1}{n_2},-\alpha-\frac{\beta_3+1}{n_2}\right)
$$
and the result follows from \eqref{eq:I_descomposicion}.
\end{proof}

\begin{remark}
It is obvious that $-\alpha\notin\mathbb{Z}$. Assume that $\frac{\beta_3+1}{n_2}\in\mathbb{Z}$. 
From~\eqref{eq:alpha_A21} and~\eqref{eq:A2}, we get a contradiction, hence $\frac{\beta_3+1}{n_2}\notin\mathbb{Z}$.
On the other side, if
$-\alpha-\frac{\beta_3+1}{n_2}\in\mathbb{Z}$, we obtain that $n_2\alpha\in\mathbb{Z}$ which is in contradiction
with~\eqref{eq:A2}. Hence, $\boldsymbol{B}\left(\frac{\beta_3+1}{n_2},-\alpha-\frac{\beta_3+1}{n_2}\right)$
is transcendental.
\end{remark}

\subsection{\texorpdfstring{Residues at poles in $A_{22}$}{Residues at poles in A22}}
\mbox{}

As in \S\ref{subsec:A12}, we perform now  a partial computation of the residue for $\alpha\in A_{22}$,
$$
\alpha=-\frac{n_2(m+n_1)+q+k}{n_2 (mn_1n_2+ q)}.
$$
From the definition of $A_{22}$ and the properties of the semigroup~$\Gamma$, we can find non-negative
integers $a',b',\ell$ are such that
$$
(a' m+b' n_1)n_2+\ell(mn_1n_2+q)=(mn_1n_2+q)n_2+k
$$
Let
$$
f_{-t}(x,y):=(x^{n_1}-y^{m})^{n_2}+ x^a y^b+t(x^{n_1}-y^{m})^{\ell}
x^{a'} y^{b'},\quad t\in\mathbb{R}_{>0}.
$$

\begin{prop}\label{prop:poles-a22}
The function
$I(f_{-t},1,1,0)(s)$ has a pole for $s=-\alpha$ 
and its residue is a polynomial of degree~$1$ in~$t$  whose
coefficient of~$t$ equals
$$
 \frac{\alpha (mn_1)^{1-\frac{\ell(\ell+1)}{n_2}}}{n_2(mn_1n_2+q)} 
\boldsymbol{B}\left(\frac{\ell+1}{n_2},-\alpha+1-\frac{\ell+1}{n_2}\right).
$$
\end{prop}

\begin{proof}
The poles we are interested in for $J_{i,j,1},J_{i,j,2}$ start, for each $i,j$, at
$$
-\frac{n_2(m +n_1+j)+qi}{n_2(mn_1n_2+q)}.
$$
For $(i,j)$ such that $n_2j+qi\leq k$ the integrals $J_{i,j,1},J_{i,j,2}$
may have poles at~$\alpha$. 
We follow the strategy of the proof of Proposition~\ref{prop:poles-a12}.
 The residues
are computed using a derivative of order $k-(n_2j+qi)$
(the steps from the first pole). It is not hard to see that if $j\neq 0$ or $i\neq 1$,
then the residues are independent of~$t$.

Let us study the behavior of $J_{1,0,1}(s)$ and $J_{1,0,2}(s)$. 
As in the proof of Proposition~\ref{prop:poles-a12}, we have
$$
\frac{\partial^k F_1^\alpha}{\partial x^k}(0,y)=\alpha k! t (mn_1)^\ell y^{q\ell}  
{F_1}^{\alpha-1} (0,y)+\dots
$$
and
\begin{gather*}
\res_{s=\alpha} J_{1,0,1}(f)(s)= \frac{q}{(mn_1n_2+q)k!}
G_{\left( \partial^{(k,0)}(F_1)^\alpha (0,\cdot) \right)}(q)=\\
t\frac{\alpha q(mn_1)^\ell}{mn_1n_2+q}
G_{\left( (mn_1)^{n_2} y^{n_2 q}+ 1 \right)}(q(\ell+1))+\dots
\end{gather*}
With the same arguments,
$$
\frac{\partial^k F_2^\alpha}{\partial y^k}(x,0)=
\alpha k! t (mn_1)^\ell x^{(n_2-\ell)q+k}  (F_2)^{\alpha-1} (x,0)+\dots
$$
and
\begin{gather*}
\res_{s=\alpha} J_{1,0,2}(f)(s)= \frac{q}{(mn_1n_2+q)k!} 
G_{\left( \partial^{(0,k)}(F_2)^\alpha (\cdot,0) \right)}(n_2(mn_1n_2\alpha+n_1+m))=\\
t\frac{\alpha q(mn_1)^\ell}{mn_1n_2+q}
G_{\left( (mn_1)^{n_2} +x^{n_2 q}\right)}(n_2(mn_1n_2\alpha+n_1+m)+(n_2-\ell)q+k)+\dots
\end{gather*}
Let us denote 
$$
s_1=\frac{mn_1n_2\alpha+n_1+m}{q}-\frac{\ell}{n_2}+\frac{k}{q n_2}+1,
\quad s_2=\frac{\ell+1}{n_2},\quad p=q n_2,\quad c=(m n_1)^{n_2}.
$$
Since $s_1+s_2=-\alpha+1$, applying Lemma~\ref{beta}, we have
\begin{equation*}
\res_{s=\alpha} J_1(s)= \frac{\alpha (mn_1)^{-\frac{\ell(\ell+1)}{n_2}} t}{n_2(mn_1n_2+q)} 
\boldsymbol{B}\left(\frac{\ell+1}{n_2},-\alpha+1-\frac{\ell+1}{n_2}\right).
\end{equation*}

\vspace*{-5mm}
\end{proof}

\begin{remark}
Note again that
$\boldsymbol{B}\left(\frac{\ell+1}{n_2},-\alpha+1-\frac{\ell+1}{n_2}\right)$ is transcendental.
\end{remark}

\section{Relation of integrals  with  Bernstein polynomial}

We are using ideas from \cite{Pi862,Pi861,Pi872}. Let us fix notations that may cover all the cases.
We fix $f,g,Y,g_Y,\mathcal{D}_Y$ with the following properties:
\begin{enumerate}
\enet{(B\arabic{enumi})}
\item The characteristic sequence of $f\in\br[x,y]$ is $(n_1n_2,mn_2,mn_2+q)$.
\item The characteristic sequence of $g\in\br[x,y]$ is $(n_1,m)$ and it has
maximal contact with $f$ among all the singularities with the same characteristic sequence.
\item The polynomial $Y(x^{\frac{1}{m}})\in\br[x^{\frac{1}{m}}]$ (where one of its $n_1$-roots is still
in $\br[x^{\frac{1}{m}}]$) satisfies one of the following conditions:
\begin{itemize}
\item $\ord_x(g(x,Y(x^{\frac{1}{m}})))>\frac{mn_1n_2+q}{mn_2}$
and it is monotonically increasing in $\br_{\geq 0}$.
\item $Y\equiv 1$.
\end{itemize}
\item $g_Y$ is as in \ref{H3} in \S\ref{sec:res2}.
\item $\mathcal{D}_Y:=\{(x,y)\in\br^2\mid 0\leq x\leq 1,0\leq y\leq Y(x^{\frac{1}{m}})\}$.
\item $f(x,y)>0$ $\forall (x,y)\in\mathcal{D}_Y\setminus\{(0,0)\}$.
\end{enumerate}

Let $\beta_1,\beta_2\in\mathbb{Z}_{\geq 1}$ and $\beta_3\in\mathbb{Z}_{\geq 0}$. Let us consider
the integral
\begin{equation}
\mathcal{I}(f,\beta_1,\beta_2,\beta_3)(s)=
\int\!\!\!\int_{\mathcal{D}_Y}  f(x,y)^s\, x^{\beta_1} y^{\beta_2}\, g_Y(x,y)^{\beta_3}\frac{dx}{x}\, \frac{dy}{y}.
\end{equation}
These integrals cover those studied in Sections~\ref{sec:res1} and~\ref{sec:res2}. For those of \S\ref{sec:res1},
we take $Y\equiv 1$ and $\beta_3=0$ (hence $g_Y$ is not longer used). If we need to distinguish them,
we will denote by $\mathcal{I}_+$ those coming from~\S\ref{sec:res1} and by
$\mathcal{I}_-$ those coming from~\S\ref{sec:res2}. For $\mathcal{I}_+$ we may drop the argument~$\beta_3$.

Let us recall the definition of Bernstein-Sato polynomial~$b_f(s)$, see the Introduction. It is the lowest-degree
non-zero polynomial satisfying the existence of an $s$-differential operator
$$
\mathbf{D}=\sum_{j=0}^N D_js^j,\quad \ D_j=\sum_{i_1+i_2<M} a_{j, i_1,  i_2}(x,y) \frac{\partial^{{i_1}}}{\partial x^{i_1}}
\frac{\partial^{{i_2}}}{\partial y^{i_2}},
 a_{j, i_1,  i_2}\in\mathbb{C}[x,y]
$$
such that
\begin{equation}\label{eq:b0}
\mathbf{D}\cdot f^{s+1}=b_f(s) f^s.
\end{equation}
Moreover, see e.g.~\cite{Co95},
if $f\in\mathbb{K}[x,y]$, $\mathbb{K}\subset\br$, the polynomials  
$a_{j, i_1,  i_2}$  have coefficients over~$\mathbb{K}$.
Applying~\eqref{eq:b0}, we have
\begin{equation}\label{eq:b-int}
\mathcal{I}(f,\beta_1,\beta_2,\beta_3)(s)=
\frac{1}{b_{f}(s)}\mathcal{J},\quad
\mathcal{J}:=\int\!\!\!\int_ {\mathcal{D}_Y}  \mathbf{D}[f(x,y)^{s+1}] x^{\beta_1} y^{\beta_2} g_Y(x,y)^{\beta_3}\frac{dx}{x} \frac{dy}{y}.
\end{equation}
Following the definition of $\mathbf{D}$, $\mathcal{J}$ is a linear combination (with coefficients in $\mathbb{K}[s]$)
of integrals
$$
\mathcal{I}_{i_1,i_2}(\beta'_1,\beta'_2,\beta_3)(s)=
\int\!\!\!\int_{\mathcal{D}_Y}\frac{\partial^{i_1+i_2}f^{s+1}(x,y)}{\partial x^{i_1}\partial y^{i_2}}
x^{\beta'_1-1}y^{\beta'_2-1} g_Y(x,y)^{\beta_3} dx dy,
$$
with $\beta'_i\geq\beta_i$.

Using \eqref{eq:derivar}, we could express these integrals using derivatives of $f$ and powers of the 
type~$f^{s+1-m}$ (for some non-negative integer~$m$). But, following the ideas in~\cite{Pi861},
we will use integration by parts in order to  do not decrease the exponent~$s+1$.

Let us define $X(y^{\frac{1}{n_1}})$ the inverse of the function $Y(x^{\frac{1}{m}})$, when $Y$ is not constant; we set $X\equiv 0$
if $Y$ is constant. Note that $X(y^{\frac{1}{n_1}})$ is an analytic function in~$y^{\frac{1}{n_1}}$
with coefficients in~$\mathbb{K}$.
The integration by parts with respect to $x$ (if $i_1>0)$ yields
\begin{gather*}
\mathcal{I}_{i_1,i_2}(\beta'_1,\beta'_2,\beta_3)(s)=U-W
\end{gather*}
where
\begin{gather*}
U=\int_0^{Y(1)}
\left[\frac{\partial^{i_1+i_2-1}f^{s+1}\left(x,y\right)}{\partial x^{i_1-1}\partial y^{i_2}}
x^{\beta'_1-1}(g_Y(x,y))^{\beta_3}\right]_{X(y^{\frac{1}{n_1}})}^{1}
y^{\beta'_2}\frac{dy}{y},\\
W=\int\!\!\!\int_{\mathcal{D}_Y}\frac{\partial^{i_1+i_2-1}f^{s+1}}{\partial x^{i_1-1}\partial y^{i_2}}\left(x,y\right)
\frac{\partial(x^{\beta'_1-1}(g_Y(x,y))^{\beta_3})}{\partial x}y^{\beta'_2}dx\frac{dy}{y}.
\end{gather*}
A similar formula is obtained with respect to~$y$.

Using again \eqref{eq:derivar}, we can see that $U$ is a linear combination
with coefficients in~$\mathbb{K}$ of integrals as in Corollary~\ref{cor:int_una_var} (where the exponents may decrease).
The term~$W$ is again a linear combination with coefficients in~$\mathbb{K}$ 
of integrals $\mathcal{I}_{i_1-1,i_2}(\beta''_1,\beta''_2,\beta'_3)(s)$.
Since the index $i_1$ decreases (and the same happens with $i_2$ integrating with respect to~$y$)
we can summarize these arguments in the following Proposition.

\begin{prop}
Let $f\in \mathbb{K}[x,y]$ be a polynomial whose local complex singularity at the origin
has
two Puiseux pairs and such that $\mathbb{K}$ is an algebraic extension of~$\mathbb{Q}$. 
If $\beta_1,\beta_2\geq 1$, and $\beta_3\geq 0$ 
then $\mathcal{I}_{i_1,i_2}(\beta_1,\beta_2,\beta_3)(s)$
is a linear combination over~$\mathbb{K}[s]$ of:
\begin{enumerate}
\enet{\rm(\arabic{enumi})}
\item  
meromorphic functions
having only simple poles whose residues are algebraic over~$\mathbb{K}$;
\item and
integrals
$\mathcal{I}(f,\beta'_1,\beta'_2,\beta'_3)(s+1)$ for some
triples $(\beta'_1,\beta'_2,\beta'_3)$ with $\beta'_i\geq \beta_i$ for $1\leq i \leq 3$.
\end{enumerate}
\end{prop}

\begin{cor}\label{cor:polos}
Let $f\in \mathbb{K}[x,y]$ be a polynomial whose local complex singularity at the origin
has
two Puiseux pairs and such that $\mathbb{K}$ is an algebraic extension of~$\mathbb{Q}$. 
Then the integral $\mathcal{I}(f,\beta_1,\beta_2,\beta_3)(s)$ is the product of $b_f(s)^{-1}$
and a linear combination over $\mathbb{K}[s]$ of meromorphic functions whose residues are algebraic over~$\mathbb{K}$
and integrals $\mathcal{I}(f,\beta'_1,\beta'_2,\beta'_3)(s+1)$.
\end{cor}

These results allow to detect roots of Bernstein polynomials in some cases.

\begin{thm}\label{pole-integral-root}
Let $f\in \mathbb{K}[x,y]$ be a polynomial whose local complex singularity at the origin
has
two Puiseux pairs and its algebraic monodromy
has distinct eigenvalues and such that $\mathbb{K}$ is an algebraic extension of~$\mathbb{Q}$. 
Let $\alpha$ be a pole of  $\mathcal{I}_{}(f,\beta_1,\beta_2,\beta_3)(s)$  with transcendental residue,  
and such that $\alpha+1$ is not a pole of $\mathcal{I}_{}(f,\beta'_1,\beta'_2,\beta'_3)(s)$ 
for any $(\beta'_1,\beta'_2,\beta'_3)$. Then $\alpha$ is a root of
the Bernstein-Sato polynomial  $b_{f}(s)$ of $f$.
\end{thm}

\begin{proof}
Let us consider the equality~\eqref{eq:b-int}. On the left-hand side of the integral, 
$\alpha$~is a pole with transcendental residue. Let us study the situation on
the right-hand side. It can be either a pole of $\mathcal{J}$ or a root of $b_f(s)$ (only simple roots!). 
Note that by Corollary~\ref{cor:polos},  
if $\alpha$ is a pole of~$\mathcal{J}$ then its residue must be algebraic. Then,
$\alpha$~must be a root of $b_f(s)$.
\end{proof}

\section{Yano's conjecture for two-Puiseux-pair singularities}

Let $(n_1n_2,mn_1,mn_2+q)$ be a characteristic sequence  such that $\gcd (q,m)=\gcd(q,n_1)=1$, i.e.,
the monodromy has distinct eigenvalues. The Bernstein-Sato polynomial of a germ~$f$ with this 
characteristic sequence, depends on~$f$, but there is a \emph{generic} Bernstein polynomial~$b_{\mu, \text{gen}}(s)$: for any versal
deformation of such an~$f$, there exists a Zariski dense open set\ $\mathcal{U}$ on which the Bernstein-Sato polynomial
of any germ in~$\mathcal{U}$ equals $b_{\mu,gen}(s)$. 

Recall that the hypothesis on the eigenvalues of the monodromy implies that
the set of $b$-exponents consists in a set of $\mu$ distinct values, which are opposite to the roots of the Bernstein polynomial, being $\mu$
the Milnor number of any irreducible germ with $(n_1n_2,mn_1,mn_2+q)$ as  characteristic sequence.  
Hence, in order to prove that Yano's Conjecture holds for those characteristic sequences, we need to prove
that the set of roots of the Bernstein polynomial~$b_{\mu, \text{gen}}(s)$ is $A_1\cup A_2$.

\begin{thm} \label{yano2pares}
Let  $f(x,y)\in \bc\{x,y\}$ be an irreducible germ of plane curve which has two Puiseux pairs and its algebraic monodromy has distinct eigenvalues. Then
Yano's Conjecture holds for generic polynomials having as characteristic sequence $(n_1n_2,mn_1,mn_2+q)$ 
such that $\gcd (q,m)=\gcd(q,n_1)=1$, that is the set of opposite $b$-exponents is $A_1\cup A_2$.
\end{thm}
 
\begin{proof} 
Let us fix an element $\alpha\in A_1\cup A_2$. 

Let us start with $\alpha\in A_1$.
Note that $\alpha +1\geq -\frac{m+n}{mn_1n_2}$, 
which is the greater abscissa of convergence of $\mathcal{I}(f,\beta'_1,\beta'_2)(s)$ for all $\beta'_1,\beta'_2$.
As a consequence, $\alpha$ satisfies the second hypothesis of Theorem~\ref{pole-integral-root}
for any~$f$ of type $(n_1n_2, mn_2,mn_2+q )^+$.

Assume that $\alpha\in A_{11}$. Let us pick-up $f$ of type $(n_1n_2, mn_2,mn_2+q )^+$
and let~$\mathcal{V}$ be the set of such polynomials.
We have proved in Proposition~\ref{prop:a11} that there exist $\beta_1,\beta_2\in\mathbb{Z}_{\geq 1}$
such that $\mathcal{I}(f,\beta_1,\beta_2)(s)$ has a simple pole for $s=\alpha$
and its residue equals (up to a rational number) $\boldsymbol{B}\left(\frac{\beta_1}{n_1},\frac{\beta_2}{m}\right)$,
and neither $\frac{\beta_1}{n_1},\frac{\beta_2}{m}$ nor its sum (which equals $-n_2\alpha$) are integers. 
As a consequence, this residue is a transcendental number, see Remark~\ref{obs:sch}.
Then, if we choose~$f$ with algebraic coefficients, all the hypotheses of 
Theorem~\ref{pole-integral-root} are fulfilled and $\alpha$ is a root of the 
Bernstein polynomial of~$f$.

Since $\mathcal{V}$ determines a non-empty open set in the real part of a versal deformation,
there is a non-empty real open set $\mathcal{V}_1$ of real polynomials whose Bernstein polynomial is~$b_{\mu,gen}(s)$.
Since polynomials with algebraic coefficients are dense, we conclude that $\alpha$ is a root of~$b_{\mu, \text{gen}}(s)$,
$\forall\alpha\in A_{11}$.

Now let us assume $\alpha \in A_{12}$.  By Proposition~\ref{prop:poles-a12}, 
we know that there is an $f_{+t}$ of type $(n_1n_2,mn_1,mn_2+q)^+ $ (and algebraic coefficients) 
such that $\mathcal{I}(f_{+t},1,1)(s)$ 
has a simple pole for $s=\alpha$ with a transcendental residue. 
As above, Theorem~\ref{pole-integral-root} ensures
that $\alpha$ is a root of the Bernstein polynomial of this particular~$f_{+t}$. 
Recall, from Lemma~\ref{lema:lct0}, that  $\forall \alpha \in A_{12}$, $\alpha+1>-\frac{m+n_1}{n_1n_2m}$, 
in particular
$\alpha+1$ cannot be a root of the Bernstein polynomial for any $f$ with characteristic sequence $(n_1n_2,mn_1,mn_2+q)$.
We are in the hypothesis of Proposition~\ref{semi};
The lower semicontinuity implies that either $\alpha$ or $\alpha+1$ are roots of~$b_{\mu, \text{gen}}(s)$, hence, 
$\alpha$ is a root of~$b_{\mu,gen}(s)$, $\forall\alpha\in A_{12}$. 
 
Once the statement is done for  the set $A_1$ we can use the same kind of arguments for the set $A_2$. 
If $\alpha\in A_2$, by \eqref{eq:A2}, $\alpha+1>\frac{(m+n_1)n_2+q}{n_2(mn_1n_2+q)}$
which is the maximum pole that can be congruent with $\alpha\bmod\bz$. This ensures
the fulfillment of the second hypothesis of Theorem~\ref{pole-integral-root}
for any~$f$ of type $(n_1n_2, mn_2,mn_2+q )^-$.
The rest of the arguments follow the same ideas as above
using instead Propositions~\ref{prop:a21} and~\ref{prop:poles-a22}.
\end{proof}

\vspace*{-10mm}

\appendix
\section{Technical proofs}\label{sec:appendix}

\makeatletter 
\renewcommand\thesection{\@Alph\c@section}
\makeatother


\begin{proof}[Proof of Proposition{\rm ~\ref{Essou1a}}.]\label{page-Essou1a}
The proof follows the same ideas as in Proposition~\ref{prop:int_una_var}.
Let us consider first the Taylor expansion of~$f^s$ with respect to~$x$:
$$
f^s(x,y)=\sum_{\nu_1=0}^{N_1}\frac{1}{\nu_1!}\frac{\partial^{\nu_1}{f^s}}{\partial x^{\nu_1}}(0,y)x^{\nu_1}+
\frac{1}{N_1!}\int_{0}^1 x^{N_1+1} (1-t_1)^{N_1} \frac{\partial^{N_1+1}{f^s}}{\partial x^{N_1+1}}(t_1 x,y)dt_1.
$$
We apply to each function above its Taylor expansion with respect to~$y$:
\begin{gather*}
f^s(x,y)=\sum_{\nu_1=0}^{N_1}\sum_{\nu_2=0}^{N_2}\frac{1}{\nu_1!\nu_2!}
\frac{\partial^{\nu_1+\nu_2}{f^s}}{\partial x^{\nu_1}\partial y^{\nu_2}}(0,0)x^{\nu_1} y^{\nu_2}+\\
\sum_{\nu_1=0}^{N_1}\frac{x^{\nu_1}}{\nu_1!N_2!}\int_{0}^1 y^{N_2+1}(1-t_2)^{N_2}
\frac{\partial^{\nu_1+N_2+1}{f^s}}{\partial x^{\nu_1}\partial y^{N_2+1}}(0,t_2 y)dt_2+\\
\sum_{\nu_2=0}^{N_2}\frac{y^{\nu_2}}{N_1!\nu_2!}\int_{0}^1 x^{N_1+1} (1-t_1)^{N_1} 
\frac{\partial^{N_1+\nu_2+1}{f^s}}{\partial x^{N_1+1}\partial y^{\nu_2}}(t_1 x,0)dt_1+
\end{gather*}
\begin{gather*}
\frac{1}{N_1!N_2!}\int_{0}^1\int_{0}^1 x^{N_1+1}y^{N_2+1} (1-t_1)^{N_1}(1-t_2)^{N_2} 
\frac{\partial^{N_1+N_2+2}{f^s}}{\partial x^{N_1+1}\partial y^{N_2+1}}(t_1 x,t_2 y)dt_1 dt_2.
\end{gather*}
Consider the following notation:
\begin{align*}
\psi_{N_1,\nu_2}^1(x,s):=&\frac{1}{N_1!\nu_2!}\int_{0}^1 (1-t_1)^{N_1} 
\frac{\partial^{N_1+\nu_2+1}{f^s}}{\partial x^{N_1+1}\partial y^{\nu_2}}(t_1 x,0)dt_1\\
\psi_{\nu_1,N_2}^2(y,s):=&\frac{1}{\nu_1!N_2!}\int_{0}^1 (1-t_2)^{N_2}
\frac{\partial^{\nu_1+N_2+1}{f^s}}{\partial x^{\nu_1}\partial y^{N_2+1}}(0,t_2 y)dt_2\\
\mathcal{S}_{N_1,N_2}(x,y,s):=&\frac{1}{N_1!N_2!}\int_{0}^1\!\!\!\int_{0}^1 \!\!  (1-t_1)^{N_1}(1-t_2)^{N_2} 
\frac{\partial^{N_1+N_2+2}{f^s}}{\partial x^{N_1+1}\partial y^{N_2+1}}(t_1 x,t_2 y)dt_1 dt_2.
\end{align*}
These functions are holomorphic for  $s\in\mathbb{C}$.
Hence, one can write
\begin{equation}\label{eq:Essouabri}
\begin{split}
\mathcal{Y}(s)= 
\sum_{\nu_1=0}^{N_1}\sum_{\nu_2=0}^{N_2}\frac{1}{\nu_1!\nu_2!}
\frac{\partial^{\nu_1+\nu_2}{f^s}}{\partial x^{\nu_1}\partial y^{\nu_2}}(0,0)
\frac{1}{(a_1 s+b_1+\nu_1)(a_2 s+b_2 +\nu_2)}+\\
\sum_{\nu_1=0}^{N_1}\frac{1}{a_1 s+b_1+\nu_1}\int_{0}^1 y^{a_2 s+b_2+N_2}\psi^2_{\nu_1,N_2}(y,s)dy+\\
\sum_{\nu_2=0}^{N_2}\frac{1}{a_2 s+b_2+\nu_2}\int_{0}^1 x^{a_1 s+b_1+N_1}\psi^1_{N_1,\nu_2}(x,s)dx+\\
\int_{0}^1\int_0^1x^{a_1 s+b_1+N_1} y^{a_2 s+b_2+N_2}\mathcal{S}_{N_1,N_2}(x,y,s)dx dy.
\end{split}
\end{equation}
Let us denote
\begin{align*}
\varphi_{a_1,b_1,\nu_2}^1(s):=&\int_{0}^1 x^{a_1 s+b_1+N_1}\psi^1_{N_1,\nu_2}(x,s)dx\\
\varphi_{a_2,b_2,\nu_1}^2(s):=&\int_{0}^1 y^{a_2 s+b_2+N_2}\psi^2_{\nu_1,N_2}(y,s)dy\\
\mathcal{R}_{a_1,b_1,a_2,b_2}(s):=&\int_{0}^1\int_0^1x^{a_1 s+b_1+N_1} y^{a_2 s+b_2+N_2}\mathcal{S}_{N_1,N_2}(x,y,s)dx dy. 
\end{align*}
The integral function $\varphi_{a_1,b_1,\nu_2}^1$ is absolutely convergent and holomorphic for $\Re(s)>-\frac{b_1+N_1+1}{a_1}$,
while $\varphi_{a_2,b_2,\nu_1}^2$ is holomorphic for $\Re(s)>-\frac{b_2+N_2+1}{a_2}$.

The function $\mathcal{R}_{a_1,b_1,a_2,b_2}$ is absolutely convergent and holomorphic for
$\Re(s)>\max\{-\frac{b_1+N_1+1}{a_1},-\frac{b_2+N_2+1}{a_2}\}$.
The result follows. 
\end{proof}

\begin{proof}[Proof of Proposition{\rm~\ref{continuation}}.]
The hypothesis ensures that the pole is simple.
Choose $N_1\geq \nu_1$ and $N_2$ 
such that $\alpha>-\frac{b_2+N_2+1}{a_2}$. We use the functions
and equalities introduced in the proof of Proposition~\ref{Essou1a}. 
The residue is obtained by evaluating $\frac{a_1 s+b_1+\nu_1}{a_1}\mathcal{Y}(s)$
at $\alpha$. Using \eqref{eq:Essouabri}, we have
\begin{equation*}
\sum_{\nu_2=0}^{N_2}\frac{1}{(a_2 \alpha+b_2 +\nu_2)a_1 \nu_1!\nu_2!}
\frac{\partial^{\nu_1+\nu_2}{f^\alpha}}{\partial x^{\nu_1}\partial y^{\nu_2}}(0,0)
+
\frac{1}{a_1}\int_{0}^1 y^{a_2\alpha+b_2+N_2}\psi^2_{\nu_1,N_2}(y,\alpha)dy
\end{equation*}
Then,
\begin{equation*}
\res_{s=\alpha} \mathcal{Y}(s)=\sum_{\nu_2=0}^{N_2}\frac{1}{(a_2 \alpha+b_2 +\nu_2)a_1 \nu_1!\nu_2!}
\frac{\partial^{\nu_1+\nu_2}{f^\alpha}}{\partial x^{\nu_1}\partial y^{\nu_2}}(0,0)
+\frac{1}{a_1}\varphi^2_{\nu_1,N_2}(\alpha).
\end{equation*}
Consider the integral
$$
  \int_0^1 \partial^{(\nu_1,0)} (f^\alpha) (0,y) y^{s}  \frac{dy}{y}
$$
The Taylor formula yields
\begin{gather*}
\partial^{(\nu_1,0)} (f^\alpha) (0,y)=
\frac{\partial^{\nu_1}f^\alpha}{\partial x^{\nu_1}}(0,y)=\\
\sum_{\nu_2=0}^{N_2}\frac{1}{\nu_2!}
\frac{\partial^{\nu_1+\nu_2}f^\alpha}{\partial x^{\nu_1}\partial y^{\nu_2}}(0,0)y^{\nu_2}+
\frac{1}{N_2!}
\int_{0}^1 y^{N_2+1} (1-t_2)^{N_2} {(f^\alpha)}^{(N_2+1)}(0,t_2 y)dt_2=\\
\sum_{\nu_2=0}^{N_2}\frac{1}{\nu_2!}
\frac{\partial^{\nu_1+\nu_2}f^\alpha}{\partial x^{\nu_1}\partial y^{\nu_2}}(0,0)y^{\nu_2}+
\nu_1!y^{N_2+1}\psi^2_{\nu_1,N_2}(y,\alpha).
\end{gather*}
We integrate that function (multiplied by $y^{s-1}$) to get
$$
\sum_{\nu_2=0}^{N_2}\frac{1}{(\nu_2+s)\nu_2!}
\frac{\partial^{\nu_1+\nu_2}f^\alpha}{\partial x^{\nu_1}\partial y^{\nu_2}}(0,0)+
\frac{1}{a_1}\int_{0}^1 y^{s+N_2}\psi^2_{\nu_1,N_2}(y,\alpha)dy
$$
and the equality holds.
\end{proof}

\begin{proof}[Proof of Lemma{\rm~\ref{beta}}]
Let $G_1:=G_{ \left(\, y^{p}+c\right)^{\alpha}}(ps_1)$,
\begin{gather*}
G_1=\! 
\int_0^1\!\!\!    \left(y^{p}+c\right)^{\alpha}\! y^{p s_1}  \frac{dy}{y}
=\frac{c^{\alpha}}{p} \int_0^1   \left(\frac{y}{c}+1\right)^{\alpha}  y^{s_1} \frac{dy}{y}
=\frac{c^{-s_2}}{p} 
\int_0^{c^{-1}} \!\!\!   \left(\, y+1\right)^{\alpha}   y^{s_1}  \frac{dy}{y}.
\end{gather*}
Let $G_2:=G_{ \left(\,1 + cx^{p} \right)^{\alpha}}(ps_2)$,
\begin{gather*}
G_2 =\!\!\! \int_0^1\!\!\!\left(\,1 + cx^{p} \right)^{\alpha}   x^{p s_2}  \frac{dx}{x}
\!=\!\frac{1}{p} \int_0^1\!\!\!    \left(\,1 +cx \right)^{\alpha}   x^{s_2}  \frac{dx}{x}=\\
\frac{1}{p} \int_1^\infty\!\!\!    \left(\, x +c\right)^{\alpha}   x^{s_1}  \frac{dx}{x}
=\frac{c^{-s_2}}{p} \int_{c^{-1}}^\infty\!\!\!    
\left(\, x +1\right)^{\alpha}   x^{s_1}  \frac{dx}{x}.
\end{gather*}
Thus:
\begin{equation*}
G_1+ G_2
=\frac{c^{-s_2}}{p} \int_0^\infty    \left(\, x +1\right)^{\alpha}   x^{s_1} \frac{dx}{x}
=\frac{c^{-s_2}}{p} \boldsymbol{B}\left( s_1, s_2\right). 
\end{equation*}

\vspace*{-5mm}
\end{proof}


\def\cprime{$'$}
\providecommand{\bysame}{\leavevmode\hbox to3em{\hrulefill}\thinspace}
\providecommand{\MR}{\relax\ifhmode\unskip\space\fi MR }
\providecommand{\MRhref}[2]{%
  \href{http://www.ams.org/mathscinet-getitem?mr=#1}{#2}
}
\providecommand{\href}[2]{#2}

\end{document}